\documentclass[12pt]{amsart}
\usepackage{mathrsfs}
\usepackage{eucal}
\usepackage{color}
\usepackage[all]{xy}
\usepackage{amssymb,amsmath}

\usepackage{setspace}

\calclayout
\newtheorem{dummy}{anything}[section]

\newtheorem*{thma}{Theorem A}

\newtheorem{lemma}[dummy]{Lemma}

\newtheorem{corollary}[dummy]{Corollary}

\theoremstyle{definition}
\newtheorem{definition}[dummy]{Definition}

  \newtheorem{remark}[dummy]{Remark}

\newcommand
{\eqncount}{\setcounter{equation}{\value{dummy}}%
\addtocounter{dummy}{1}}

\newcommand{\cF}{\mathcal F}

\newcommand{\bC}{\mathbf C}
\newcommand{\bD}{\mathbf D}

\newcommand{\bZ}{\mathbb Z}

\newcommand{\bbR}{\mathbb R}

\DeclareMathOperator{\Map}{Map}

\newcommand{\cy}[1]{\bZ/{#1}}
\newcommand{\la}{\langle}
\newcommand{\ra}{\rangle}

\newcommand{\vv}{\, | \,}

\def\G{\varGamma}
\def\ZG{\bZ\G}

\DeclareMathOperator{\invlim}{\lower 6pt\hbox{$\stackrel{\displaystyle{\lim}}{\leftarrow}$}}

\DeclareMathOperator{\Or}{Or}
\newcommand\OrG{\Or _{\cF}G}
\newcommand\CC{\bC}
\newcommand\DD{\bD}
\newcommand\uZ[1]{\bZ[{#1}^{\, \textbf{?}\,}]}
\newcommand\un{\underline}
\newcommand\uC[1]{\CC({#1}^{\textbf{?}})}

 \DeclareMathOperator{\Sing}{Sing}

\begin{document}

\title{A remark about dihedral group actions on spheres}
\author{Ian Hambleton}

\address{Department of Mathematics, McMaster University,  Hamilton, Ontario L8S 4K1, Canada}

\email{hambleton@mcmaster.ca }

\date{April 22, 2010}

\thanks{Research partially supported by NSERC Discovery Grant A4000. The author would like  to thank the Max Planck Institut f\"ur Mathematik in Bonn for its hospitality and support while this work was done.}

\begin{abstract} We show that a finite dihedral group does not act pseudofreely and locally linearly on an even-dimensional sphere $S^{2k}$, with $k>1$. This answers a question of R.~S.~Kulkarni from 1982.
\end{abstract}

\maketitle

\section{Introduction}
\label{sect: introduction}
In this note we let $D_p = \la a,b\vv a^p=b^2 = 1, bab = a^{-1}\ra$ denote the finite dihedral group of order $2p$, for $p$ an odd prime. A famous theorem of Milnor \cite{milnor2} states that a finite dihedral group can not act freely on a topological $n$-manifold with the mod 2 homology of $S^n$.  More generally, a \emph{pseudofree} action is one which is free outside of a discrete set of points. 
In \cite[Theorem 7.4]{kulkarni1982},  R.~S.~Kulkarni studied  orientation-preserving, pseudofree actions of finite groups $G$ on manifolds which are $\bZ/2$-homology $n$-spheres,  and found that for $n=2k$ the group $G$ must be (i) a periodic group which acts freely on $S^{2k-1}$, (ii) dihedral, or (iii)  tetrahedral, octahedral or icoshedral (when $k=1$).  The first case occurs as the suspension of any free action of a periodic group on $S^{2k-1}$, and the other cases already appear for orthogonal actions on $S^2$. Kulkarni asked whether the second case could actually occur on $S^{2k}$ if $k>1$. This turns out to be impossible.
\begin{thma} The dihedral group $G = D_p$, $p$ an odd prime, can not act pseudofreely and locally linearly on $S^{2k}$, preserving the orientation, for $k>1$.
\end{thma}
For $k$ even, we show that there does not even exist a finite pseudofree $G$-CW complex $X \simeq S^{2k}$, with $X^G=\emptyset$. For all \emph{odd} integers $k\geq 1$, such complexes do exist, for example by taking the join of $S^2$ with the action given by $G\subset SO(3)$ and a finite  Swan complex for $G$ (see \cite{swan1}, \cite{galovich-reiner-ullom1}).
\begin{remark}
My interest in this question was prompted by the recent paper of A.~Edmonds \cite{edmonds4}, where he proves this result for $k$ even. Our methods seem rather different. The discussion by Edmonds in \cite[4.1]{edmonds4} combined with Theorem A shows that there are no effective pseudo-free dihedral actions on $S^n$, for $n>2$, even if some elements of $G$ are allowed to reverse orientation.
\end{remark}
\section{The chain complex}
In this section we let $G= D_p$ and suppose that $X$ is a finite $G$-CW complex such that $X\simeq S^{2k}$, with $k>0$, and $X^G=\emptyset$. We further assume the $G$-action is pseudofree, and induces the identity on homology. It follows from \cite[Prop.~7.3]{kulkarni1982} that every non-identity element of $G$ fixes exactly two points. We assume that $X^G = \emptyset$ since this is a necessary condition for a locally-linear, pseudo-free action on a sphere (by Milnor's theorem).

Let $\CC = \uC{X})$ denote the chain complex of $X$ over the orbit category 
$\ZG:=\bZ\OrG$ with respect to the family $\cF$ of all proper subgroups of $G$ (see \cite{tomDieck2} or \cite{lueck3} for this theory). The notation means that $\CC_i(G/U) = C_i(X^U)$, for $U \leq G$, and the action of $N_G(U)/U$ on $\CC_i(G/U)$ induced by the $G$-action on $X$ is expressed algebraically through the functorial properties of $\CC$.. 

\smallskip
Our pseudo-free  assumption on the $G$-CW complex $X$  implies that $\CC_i(G/U) =  0$, if $U\neq 1$ is a non-trivial subgroup of $G$, and $i>0$. Therefore,
\eqncount
\begin{equation}\label{one} 
H_i(\CC)(G/U) = 0, \ \text{if\ }
 i > 0, \ \text{for all\ } U \neq 1.
\end{equation}
From the homology of $S^{2k}$, we have
\eqncount
\begin{equation}\label{two}
H_0(\CC)(G/1) =  \bZ, \ \text{and\ } H_i(\CC)(G/1)= 0\  \text{for\ }i \neq 0, 2k.
\end{equation}
In addition, since we assumed that $G$ acts trivially on the homology of $S^{2k}$, we have
 \eqncount
\begin{equation}\label{three}
H_{2k}(\CC)(G/1)=\bZ, \text{with trivial\ }G\text{-action}.
\end{equation} 

Let $H = \la a \ra$ and $K = \la b \ra$ denote particular subgroups of $G$, of order $p$ and $2$ respectively. The orbit types give the chain group
$$\CC_0 = \uZ{G/H} \oplus \uZ{G/K} \oplus \uZ{G/K},$$
where $\uZ{G/V}$ denotes the free right module over the orbit category with values
$$\uZ{G/V}(G/U) = \bZ\Map_G(G/U, G/V),$$ for all proper subgroups $U\leq G$. In particular, the homology of the fixed sets is given by
\eqncount
\begin{equation}\label{four}
H_0(\CC)(G/H) = \uZ{G/H}(G/H) = \bZ[N_G(H)/H] = \bZ[\cy 2],
\end{equation} 
 and 
 \eqncount
\begin{equation}\label{five}
H_0(\CC)(G/K) =\big (\uZ{G/K}(G/K)\big )^2 = \big (\bZ[N_G(K)/K]\big )^2 = \bZ\oplus \bZ.
\end{equation} 

\smallskip
\begin{definition}
A finite $\ZG$-chain complex $\CC$ of finitely-generated free $\ZG$-modules, which satisfies the algebraic conditions (\ref{one})--(\ref{five}), is called a \emph{pseudofree $\ZG$-chain complex with the $\bZ$-homology of $S^{2k}$}. 
\end{definition}
 
One example of such a complex arises from the standard orthogonal action $Y = S(V )$ of the dihedral group on $S^2$ (for $G$ as a subgroup of $SO(3)$). The $SO(3)$-representation $V = W \oplus  \bbR_{-}$ is the sum of the standard $2$-dimensional real representation $W$  (given by the action on a regular $2p$-gon in the plane), and the non-trivial $1$-dimensional  real representation $\bbR_-$.
The chain complex $\DD = \uC{Y}$ over the orbit category has the form
$${\xymatrix{  \big (\uZ{G/1}\big )^2 \ar[r]\ar@{=}[d] & \big (\uZ{G/1} \big )^3 \ar[r]\ar@{=}[d]& \uZ{G/H}\oplus
 \big  (\uZ{G/K} \big )^2\ar@{=}[d]\cr
 \DD_2\ar[r]&\DD_1\ar[r]&\DD_0}}$$
 where $H_2(\DD) = \un{\bZ}_0$ is the $\ZG$-module with value $\un{\bZ}_0(G/1) =\bZ$,  and zero otherwise. The module $H_0:= H_0(\DD)$ has value
 $H_0(G/1) =\bZ$, and values at $G/H$ and $G/K$ as listed above. In general, for any 
pseudofree $\ZG$-chain complex $\CC$ with the $\bZ$-homology of $S^{2k}$, we have $H_{2k}(\CC) = \un{\bZ}_0$ and $H_0(\CC) = H_0(\DD)$.
\begin{lemma}\label{2-skeleton}
Suppose that $\CC$ is a pseudofree $\ZG$-chain complex with the $\bZ$-homology of $S^{2k}$.  Then 
 the complex $\CC$ is chain homotopy equivalent to a finite free $2k$-dimensional chain complex $\CC'$, with $\CC'_i =  \CC_i$ for $i \geq 4$,
  whose initial part
$\CC'_2 \to \CC'_1 \to \CC'_0$ is chain isomorphic to $\DD$. 
\end{lemma}
\begin{proof}
Since  $H_0(\CC) = H_0(\DD)$, this  follows from the version of Schanuel's Lemma over the orbit category given in the proof of \cite[Lemma 8.12]{hpy1}. By construction $\CC'_3 = \CC''_3 \oplus F$, where $\CC''_3$ is isomorphic to a direct sum of copies of $\uZ{G/1}$.
\end{proof}
An immediate consequence is the statement of Theorem A for $k$ even. 

\begin{corollary}[Edmonds \cite{edmonds4}] Let $G= D_p$.  If $k$ is even, there is no effective pseudofree  $G$-action on a finite $G$-CW complex $X\simeq S^{2k}$, inducing the identity on homology.
\end{corollary}
\begin{proof} Let $\CC = \uC{X}$ denote the chain complex over the orbit category of such an action.  From the chain equivalent complex $\CC' \simeq \CC$ we can extract an exact periodic resolution
$$0 \to  \un{\bZ}_0 \to \CC_{2k} \to \CC_{2k-1} \to \dots \to \CC_4 \to  \CC''_{3} \to \un{\bZ}_0 \to 0$$
 since $H_2(\DD) = H_{2k}(\CC) = \un{\bZ}_0$. By evaluating at $G/1$ we obtain a periodic resolution  from $\bZ$ to $\bZ$  over $\bZ G$ of length $(2k-2)$. Since $G = D_p$ has periodic cohomology of period 4 (and not two), we conclude that $k$ is odd.
\end{proof}

\begin{proof}[The proof of Theorem A, $k$ odd]
Suppose, if possible, that we have a locally linear and orientation-preserving pseudofree topological action of $G$ on $S^{2k}$, for some odd integer $k \geq 3$. Then there exists a finite $G$-CW complex $X\simeq S^{2k}$, and a chain homotopy equivalence $\uC{X} \simeq \CC'$ provided by Lemma \ref{2-skeleton}. We may identify the singular set $\Sing(X)$ of $X$ with the singular set of the given action on $S^{2k}$. Let $\{x_0, x_1, x_2\}\subset \Sing(X)$ denote representatives of the distinct $G$-orbits of singular points (with $G_{x_0} = H$, and $G_{x_i} = K$ for $i=1,2$).
Around each singular point $x_i$, $0 \leq i \leq 2$, we can choose a linearly embedded $2$-disk slice $G\times_{G_{x_i}} D^2\subset S^{2k}$, since the action $(S^{2k}, G)$ is locally linear. This gives a $G$-equivariant embedding
$$f_0\colon \bigcup_{0 \leq i \leq 2}\big (G\times_{G_{x_i}} D^2\big )\subset S^{2k}.$$
Since the pseudofree orbit structure of the standard $G$-action on $S^2= S(V)$ is the same for any locally linear action on $S^{2k}$, we can consider $f_0$ to be a $G$-equivariant embedding of a tubular neighbourhood of the singular set of $S(V)$ into $S^{2k}$.
 By obstruction theory, and since  $k\geq 3$, we can extend this embedding $f_0$ to a $G$-equivariant embedding $f\colon S(V) \subset S^{2k}$. Non-equivariantly such an embedding of $S^2 \subset S^{2k}$ is isotopic to a standard embedding. We have thus obtained a dihedral action on $S^{2k}$ of the type considered in my earlier joint work with Erik Pedersen \cite{hp1}, namely one conjugate to  ``a topological action on a sphere which 
is free off a standard proper subsphere, and given by a $S(V)$ on the subsphere".
However, we proved in \cite[Theorem 7.11]{hp1} that such an action exists if and only if the representation $V$ on the subsphere contains two $\bbR_{-}$ factors. Since this is not the case for the standard $SO(3)$-representation $V$ of $G$, we conclude that a pseudofree $G$-action on $S^{2k}$ does not exist for $k > 1$.
\end{proof}

\providecommand{\bysame}{\leavevmode\hbox to3em{\hrulefill}\thinspace}
\providecommand{\MR}{\relax\ifhmode\unskip\space\fi MR }
\providecommand{\MRhref}[2]{%
  \href{http://www.ams.org/mathscinet-getitem?mr=#1}{#2}
}
\providecommand{\href}[2]{#2}

\end{document}